\documentclass[a4paper,twoside]{article}
\usepackage{amsmath,amssymb,amsfonts,amsthm}
\pdfoutput=1
\usepackage{hyperref}
\begin{document}
\title{Some recent work in  Fr\'{e}chet geometry\footnote{Invited paper,
  International Conference on Differential Geometry and Dynamical Systems, Bucharest 6-9 October 2011}}
\author{ C.T.J. Dodson}
\pagestyle{myheadings}
\markboth{Some recent work in  Fr\'{e}chet geometry}{C.T.J. Dodson}
\date{}
\maketitle

\newtheorem{definition}{Definition}[section]
\theoremstyle{theorem}
\newtheorem{theorem}{Theorem}[section]
\newtheorem{proposition}[theorem]{Proposition}
\newtheorem{corollary}[theorem]{Corollary}
\newtheorem{lemma}[theorem]{Lemma}
\newtheorem{remark}[theorem]{Remark}
\begin{abstract}
\noindent Some recent work in Fr\'{e}chet geometry is briefly reviewed. In particular an earlier result on the structure of second tangent bundles in the finite dimensional case was extended to infinite dimensional Banach manifolds and Fr\'{e}chet manifolds that could be represented as projective limits of Banach manifolds. This led to further results concerning the characterization of second tangent bundles and differential equations in the more general Fr\'{e}chet structure needed for applications. A summary is given of recent results on hypercyclicity of operators on Fr\'{e}chet spaces.
\end{abstract}
\section{Introduction}
Dodson and Radivoiovici~\cite{Dod-Rad,DodT} proved that in the case of a finite $n$%
-dimensional manifold $M$, a vector bundle structure on $T^{2}M$ can be well
defined if and only if $M$ is endowed with a linear connection:
$T^{2}M$ becomes then and only then a vector bundle over $M$ with
structure group the general linear group $GL(2n;\mathbb{R}).$  The manifolds $M$ that admit
linear connections are precisely the paracompact ones. Manifolds with
connections form a full subcategory $Man\nabla $ of the category $Man$ of
smooth manifolds and smooth maps; the constructions in
the above theorems~\cite{Dod-Rad} provide a functor
$Man\nabla \longrightarrow VBun$~\cite{DodT}.
A linear connection is a splitting of $TLM,$ which then induces splitting in
the second jet bundle $J^{2}M$ (called a \textit{dissection}
by Ambrose et al.~\cite{Ambrose}) and we get also a corresponding splitting in
$T^{2}L^{2}M.$

Dodson and Galanis~\cite{Dod-Gal1} extended the results to manifolds $M$ modeled on
an arbitrarily chosen Banach space $\mathbb{E}$. Using the Vilms~\cite{VI}
point of view for connections on infinite dimensional vector bundles and a
new formalism, it was proved that $T^{2}M$ can be thought of as a Banach vector bundle over $M$
with structure group $GL(\mathbb{E}\times \mathbb{E})$ if and only if $M$
admits a linear connection. The case of non-Banach Fr\'{e}chet
modeled manifolds was investigated~\cite{Dod-Gal1} but there are intrinsic difficulties with Fr\'{e}chet spaces. These include pathological general linear groups, which do not even admit reasonable
topological group structures. However, every Fr\'{e}chet space admits representation as
a projective limit of Banach spaces and under certain conditions this can persist into
manifold structures. By restriction to those Fr\'{e}chet
manifolds which can be obtained as projective limits of Banach manifolds~\cite{Gal1},
it is possible to endow $T^{2}M$ with a vector
bundle structure over $M$ with structure group a new topological group, that in a
generalized sense is of Lie type. This construction is equivalent to the existence
on $M$ of a specific type of linear connection characterized by a
generalized set of Christoffel symbols. We outline the methodology and a range of results in
subsequent sections but first we mention what makes the Fr\'{e}chet case important but difficult.

In a number of cases that have significance in global analysis and physical field theory, Banach space representations break down and we need Fr\'{e}chet spaces, which have weaker requirements for their topology,  see for example Clarke~\cite{Clarke} for the metric geometry of the Fr\'{e}chet manifold of all $C^\infty$ Riemannian metrics on a fixed closed finite-dimensional orientable manifold. For background to the theory see Hamilton~\cite{Hamilton} and Neeb~\cite{Neeb}, Steen and Seebach~\cite{Steen}. However, there is a price to pay for these weaker structural constraints: Fr\'{e}chet spaces lack a general solvability theory of differential equations, even linear ones; also, the space of continuous linear mappings drops out of the category while the space of linear isomorphisms does not admit a reasonable Lie group structure. We shall see that these shortcomings can be worked round to a certain extent. The developments described in this short review will be elaborated in detail in the forthcoming monograph by Dodson, Galanis and Vassilliou~\cite{DGV3}.

\subsection{Fr\'{e}chet spaces}
A {\em seminorm} \label{seminorm} on (eg for definiteness a real) vector space $X$ is a map $p: X \rightarrow \mathbb{R}$ such that
\begin{gather*}
p(x)\geq 0, \tag{i}\\
p(x+y)\leq p(x)+p(y),\tag{ii}\\
p(\lambda x)=\left\vert \lambda \right\vert p(x), \tag{iii}
\end{gather*}
for every $x,y\in X$ and $\lambda\in \mathbb{R}$.

A family of seminorms $\Gamma =\{p_{\alpha }\}_{\alpha \in I}$ on $X$ \label{smfamily} defines a unique topology $\mathcal{T}_{\Gamma }$\label{topolG} compatible with the vector structure of $X$. The neighborhood base $\mathcal{B}_{\Gamma}$ of $\mathcal{T}_{\Gamma }$   is determined by defining  
\begin{eqnarray*}
  \nonumber 
  S(\Delta ,\varepsilon )&=&\left\{ x\in \mathbb{F} : p(x)<\varepsilon,\;
                                      \forall \;p\in \Delta \right\}\\
\mathcal{B}_{\Gamma}&=&\left\{ S(\Delta ,\varepsilon ):\varepsilon >0\text{ and }\Delta
\text{ a finite subset of }\Gamma \right\} .                                      
\label{baseG}
\end{eqnarray*}

The topology $\mathcal{T}_{\Gamma }$ induced on $X$ by $p$ is the largest making all the seminorms continuous but it is not necessarily Hausdorff. In fact $(X,\mathcal{T}_{\Gamma })$ is a locally convex topological vector space and the local convexity of a topology on $X$ is its subordination to a family of seminorms.
Hausdorffness requires the further property
\[
x=0\Leftrightarrow p(x)=0,\;\,\forall p\in \Gamma.
\]
Then it is metrizable if and only if the family of seminorms is countable.

Convergence of a sequence $( x_{n})_{n\in \mathbb{N}}$ in $X$ is dependent on all the seminorms of $\Gamma$
 \[
  x_n\rightarrow x  \Leftrightarrow \ p(x_n-x)\rightarrow 0, \ \forall p \in \Gamma.
\]
Completeness is if and only if we have convergence in $X$ of every sequence $(x_{n})_{n\in \mathbb{N}}$ in $X$ with
\[
   \lim_{n.m\rightarrow \infty} p(x_{n}-x_{m}) = 0; \;\; \forall\; p \in \Gamma .
\]

\begin{definition}
A {\em Fr\'echet space} is a topological vector space $\mathbb{F}$\label{Frspace} that is locally convex, Hausdorff, metrizable and complete.
\end{definition}
So, every Banach space is a Fr\'echet space, with just one seminorm and that one is a norm.
More interesting examples include the following:
\begin{itemize}
\item The space $\mathbb{R}^{\infty }=\underset{n\in \mathbb{N}}{\prod
}\mathbb{R}^{n}$, endowed with the cartesian topology, is a Fr\'echet
space with corresponding family of seminorms
\[
\left\{p_{n}(x_{1},x_{2},...)=\left\vert x_{1}\right\vert +\left\vert
x_{2}\right\vert +...+\left\vert x_{n}\right\vert \right\} _{n\in \mathbb{N}}.
\]
Metrizability can be established by putting
\begin{equation}\label{distRinfty}
    d(x,y)=\sum_i \frac{|x_i-y_i|}{2^i(1+|x_i-y_i|)}.
\end{equation}
In $\mathbb{R}^\infty$ the completeness is inherited from that of each copy of the real line. For if $x=(x_i)$ is Cauchy in $\mathbb{R}^\infty$ then for each $i,$ $(x_i^m), m\in \mathbb{N}$ is Cauchy in $\mathbb{R}$ and hence converges, to $X_i$ say, and $(X_i)=X\in \mathbb{R}^\infty$ with $d(x_i,X_i)\rightarrow 0$ as $i\rightarrow \infty.$ Separability arises from the countable dense subset of elements having finitely many rational components and the remainder zero; second countability comes from metrizability.
Hausdorfness implies that a compact subset of a Fr\'{e}chet space is closed; a closed subspace is a Fr\'{e}chet space and a quotient by a closed subspace is a Fr\'{e}chet space.
In fact, $\mathbb{R}^{\infty}$ is a special case from a classification for Fr\'{e}chet spaces~\cite{Neeb}. For each seminorm $p_n=|| \ ||_n$ we can define the normed subspace $F_n=F/p_n^{-1}(0)$ by factoring out the null space of $p_n.$ Then, the seminorm requirement (\ref{seminorm}) provides a linear injection into the product of normed spaces
\begin{equation}\label{injinprod}
    p: F\rightarrow \prod_{n\in \mathbb{N}} F_n :f\mapsto (p_n(f))_{n\in \mathbb{N}}
\end{equation}
and the completeness of $F$ is equivalent to the closedness of $p(V)$ in the Banach product of the closures $\overline{F_n}$ and $p$ extends to an embedding of $F$ in this product. This embedding can be used to construct limiting processes for geometric structures of interest in Fr\'{e}chet manifolds modelled on $F.$
\item More generally, any \emph{countable} cartesian product of Banach spaces
$\mathbb{F}=\prod_{n\in \mathbb{N}}\mathbb{E}^{n}$ is a Fr\'echet space with topology defined by the seminorms  $(q_n)_{n\in\mathbb{N}}$, given by
\[
 q_n(x_{1},x_{2},...)=\sum_{i=1}^n\left\Vert
x_i\right\Vert _i,
\]
where $\left\Vert\, \cdot\, \right\Vert _i$ denotes the norm of the $i$-factor $\mathbb{E}^i$.
\item The space of continuous functions $C^{0}(\mathbb{R},\mathbb{R})$ is a Fr\'echet space with seminorms $(p_n)_{n\in\mathbb{N}}$ defined by
\[
   p_{n}(f)=\sup \big\{ \left\vert f(x)\right\vert ,\;x\in
                                                  \lbrack -n,n]\big\}.
\]
\item The space of smooth functions $C^{\infty}(I,\mathbb{R})$, where
$I$ is a compact interval of $\mathbb{R}$, is a Fr\'echet space with
seminorms defined by
\[
 p_{n}(f)=\overset{n}{\underset{i=0}{\sum }}\sup \big\{ \left\vert
D^{i}f(x)\right\vert ,\;x\in I\big\}.
\]
\item The space $C^{\infty}(M,V),$  of smooth sections of the vector bundle $V$ over compact smooth Riemannian manifold $M$ with covariant derivative $\nabla,$ is a
Fr\'{e}chet space with
\begin{equation}\label{CMV}
    ||f||_n = \sum_{i=0}^n sup_x|\nabla^if(x)|, \ \ {\rm for} \ n\in \mathbb{N}.
\end{equation}
\item Fr\'{e}chet spaces of sections arise
naturally as configurations of a physical field. Then the moduli space, consisting of inequivalent
configurations of the physical field, is the quotient of the
infinite-dimensional configuration space $\mathcal{X}$ by the appropriate
symmetry gauge group. Typically, $\mathcal{X}$ is modelled on a Fr\'{e}chet
space of smooth sections of a vector bundle over a closed manifold.
For example, see Omori~\cite{Omori1,Omori4}.
 \end{itemize}

\section{Banach second tangent bundle}
Let $M$ be a $C^{\infty}-$manifold modeled on a Banach space $\mathbb{E}$
and $\{(U_{\alpha },\psi _{\alpha })\}_{\alpha \in I}$ a corresponding
atlas. The latter gives rise to an atlas $\{(\pi _{M}^{-1}(U_{\alpha }),\Psi
_{\alpha })\}_{\alpha \in I}$ of the tangent bundle $TM$ of $M$ with%
\begin{equation*}
\Psi _{\alpha }:\pi _{M}^{-1}(U_{\alpha })\longrightarrow \psi _{\alpha
}(U_{\alpha })\times \mathbb{E}:[c,x]\longmapsto (\psi _{\alpha }(x),(\psi
_{\alpha }\circ c)^{\prime }(0)),
\end{equation*}%
where $[c,x]$ stands for the equivalence class of a smooth curve $c$ of $M$
with $c(0)=x$ and
$$(\psi _{\alpha }\circ c)^{\prime }(0)=[d(\psi _{\alpha
}\circ c)(0)](1).$$
The corresponding trivializing system of $T(TM)$ is
denoted by
$$\{(\pi _{TM}^{-1}(\pi _{M}^{-1}(U_{\alpha })),\widetilde{\Psi }%
_{\alpha })\}_{\alpha \in I}.$$
Adopting the formalism of Vilms~\cite{VI}, a connection on $M$ is a vector
bundle morphism:
\begin{equation*}
\nabla:T(TM)\longrightarrow TM
\end{equation*}
with the additional property that the mappings $\omega _{\alpha }:\psi
_{\alpha }(U_{\alpha })\times \mathbb{E}\rightarrow \mathcal{L}(\mathbb{E},\mathbb{E)}$
defined by the local forms of $\nabla:$
\begin{equation*}
\nabla_{\alpha }:\psi _{\alpha }(U_{\alpha })\times \mathbb{E}\times \mathbb{E}%
\times \mathbb{E}\rightarrow \psi _{\alpha }(U_{\alpha })\times \mathbb{E}
\end{equation*}%
with $\nabla_{\alpha }:=\Psi _{\alpha }\circ \nabla\circ (\widetilde{\Psi }_{\alpha
})^{-1},$ $\alpha \in I,$ via the relation
\begin{equation*}
\nabla_{\alpha }(y,u,v,w)=(y,w+\omega _{\alpha }(y,u)\cdot v),
\end{equation*}%
are smooth. Furthermore, $\nabla$ is a linear connection on $M$ if and only if $%
\{\omega _{\alpha }\}_{\alpha \in I}$ are linear with respect to the second
variable.

Such a connection $\nabla$ is fully characterized by the family of Christoffel
symbols $\{\Gamma _{\alpha }\}_{\alpha \in I}$ , which are smooth mappings%
\begin{equation*}
\Gamma _{\alpha }:\psi _{\alpha }(U_{\alpha })\longrightarrow \mathcal{L}(%
\mathbb{E},\mathcal{L}(\mathbb{E},\mathbb{E}))
\end{equation*}%
defined by $\Gamma _{\alpha }(y)[u]=\omega _{\alpha }(y,u)$, $(y,u)\in \psi
_{\alpha }(U_{\alpha })\times \mathbb{E}$.

The requirement that a connection is well defined on the common areas of
charts of $M$, yields the Christoffel symbols satisfying the following
compatibility condition:
\begin{equation}
\begin{array}{c}
\Gamma _{\alpha }(\sigma _{\alpha \beta }(y))(d\sigma _{\alpha \beta
}(y)(u))[d(\sigma _{\alpha \beta }(y))(v)]+(d^{2}\sigma _{\alpha \beta
}(y)(v))(u)= \\
=d\sigma _{\alpha \beta }(y)((\Gamma _{\beta }(y)(u))(v)),
\end{array}
\end{equation}
for all $(y,u,v)\in \psi _{\alpha }(U_{\alpha }\cap U_{\beta })\times
\mathbb{E}\times \mathbb{E}$, and $d$, $d^{2}$ stand for the first and the
second differential respectively. Here by $\sigma _{\alpha \beta }$ we
denote the diffeomorphisms $\psi _{\alpha }\circ \psi _{\beta }^{-1}$ of $\mathbb{E}$.
For further details and the relevant proofs see~\cite{VI}.

Let $M$ be a smooth manifold modeled on the Banach
space $\mathbb{E}$ and $\{(U_{\alpha },\psi _{\alpha })\}_{\alpha \in I}$ a
corresponding atlas. For each $x\in M$ we define the following equivalence
relation on $C_{x}=\{f:(-\varepsilon ,\varepsilon )\rightarrow M$ $|$ $f$
smooth and $f(0)=x$, $\varepsilon >0\}$:%
\begin{equation}
f\approx _{x}g\Leftrightarrow f^{^{\prime }}(0)=g^{\prime }(0)\text{ and }
f^{\prime \prime }(0)=g^{\prime \prime }(0),
\end{equation}%
where by $f^{^{\prime }}$ and $f^{^{\prime \prime }}$ we denote the first
and the second, respectively, derivatives of $f$:%
\begin{eqnarray*}
f^{\prime } &:&(-\varepsilon ,\varepsilon )\rightarrow TM:t\longmapsto
\lbrack df(t)](1) \\
f^{\prime \prime } &:&(-\varepsilon ,\varepsilon )\rightarrow
T(TM):t\longmapsto \lbrack df^{\prime }(t)](1).
\end{eqnarray*}

The \textit{tangent space of order two} of $M$ at the point $x$
is the quotient $T_{x}^{2}M=C_{x}/\approx _{x}$ and the \textit{tangent
bundle of order two} of $M$ is the union of all tangent spaces of order 2:
$T^{2}M:=\underset{x\in M}{\cup }T_{x}^{2}M$.
Of course, $T_{x}^{2}M$ can be thought of as a
topological vector space isomorphic to $\mathbb{E}\times \mathbb{E}$ via the
bijection
\begin{equation*}
T_{x}^{2}M\overset{\simeq }{\longleftrightarrow }\mathbb{E}\times \mathbb{E}%
:[f,x]_{2}\longmapsto ((\psi _{\alpha }\circ f)^{\prime }(0),(\psi _{\alpha
}\circ f)^{\prime \prime }(0)),
\end{equation*}%
where $[f,x]_{2}$ is the equivalence class of $f$ with respect to
$\approx _{x}$. However, this structure depends on the choice of the chart
$(U_{\alpha },\psi _{\alpha })$, hence a definition of a vector bundle
structure on $T^{2}M$ cannot be achieved by the use of the aforementioned
bijections. The most convenient way to overcome this obstacle is to assume
that the manifold $M$ is endowed with the additional structure of a linear
connection.

\begin{theorem}
\label{T2Mvb} For every linear connection $\nabla$ on the
manifold $M$, $T^{2}M$ becomes a Banach vector bundle with structure
group the general linear group $GL(\mathbb{E}\times \mathbb{E)}$.
\end{theorem}

\begin{proof}
Let $\pi _{2}:T^{2}M\rightarrow M$ be the natural projection of $T^{2}M$ to $%
M$ with $\pi _{2}([f,x]_{2})=x$ and $\{\Gamma _{\alpha }:\psi _{\alpha
}(U_{\alpha })\longrightarrow \mathcal{L}(\mathbb{E},\mathcal{L}(\mathbb{E},%
\mathbb{E}))\}_{a\in I}$ the Christoffel symbols of the connection $D$ with
respect to the covering $\{(U_{a},\psi _{a})\}_{a\in I}$ of $M$. Then, for
each $\alpha \in I$, we define the mapping $\Phi _{\alpha }:\pi
_{2}^{-1}(U_{\alpha })\longrightarrow U_{\alpha }\times \mathbb{E}\times
\mathbb{E}$ with
\begin{equation*}
\Phi _{\alpha }([f,x]_{2})=(x,(\psi _{\alpha }\circ f)^{\prime }(0),(\psi
_{\alpha }\circ f)^{\prime \prime }(0)+\Gamma _{\alpha }(\psi _{\alpha
}(x))((\psi _{\alpha }\circ f)^{\prime }(0))[(\psi _{\alpha }\circ
f)^{\prime }(0)]).
\end{equation*}%
These are obviously well defined and injective mappings. They are also
surjective since every element $(x,u,v)\in U_{\alpha }\times \mathbb{E}\times
\mathbb{E}$ can be obtained through $\Phi _{\alpha }$ as the image of the
equivalence class of the smooth curve%
\begin{equation*}
f:\mathbb{R}\rightarrow \mathbb{E}:t\mapsto \psi _{\alpha }(x)+tu+\frac{t^{2}%
}{2}(v-\Gamma _{\alpha }(\psi _{\alpha }(x))(u)[u]),
\end{equation*}%
appropriately restricted in order to take values in $\psi _{\alpha
}(U_{\alpha })$. On the other hand, the projection of each $\Phi _{\alpha }$
to the first factor coincides with the natural projection $\pi _{2}:$ $%
pr_{1}\circ \Phi _{\alpha }=\pi _{2}$. Therefore, the trivializations $%
\{(U_{\alpha },\Phi _{\alpha })\}_{a\in I}$ define a fibre bundle
structure on $T^{2}M$ and we need now to focus on the behavior of
the mappings $\Phi _{\alpha }$ on areas of $M$ that are covered by
common domains of different charts. Indeed, if $(U_{\alpha },\psi
_{\alpha }),$ $(U_{\beta },\psi _{\beta })$ are two such charts,
let $(\pi _{2}^{-1}(U_{\alpha }),\Phi _{\alpha })$, $(\pi
_{2}^{-1}(U_{\beta }),\Phi _{\beta })$ be the corresponding
trivializations of $T^{2}M$. Taking into account the compatibility
condition (1) satisfied by the Christoffel symbols $\{\Gamma
_{\alpha }\}$ we see that:%
\begin{equation*}
(\Phi _{\alpha }\circ \Phi _{\beta }^{-1})(x,u,v)=\Phi _{\alpha }([f,x]_{2}),
\end{equation*}%
where $(\psi _{\beta }\circ f)^{\prime }(0)=u$ and $(\psi _{\beta }\circ
f)^{\prime \prime }(0)+\Gamma _{\beta }(\psi _{\beta }(x))(u)[u]=v$. As a
result,%
\begin{equation*}
(\Phi _{\alpha }\circ \Phi _{\beta }^{-1})(x,u,v)=
\end{equation*}%
\begin{equation*}
((\psi _{\alpha }\circ \psi _{\beta }^{-1})(\psi _{\beta }(x)),d(\psi
_{\alpha }\circ \psi _{\beta }^{-1}\circ \psi _{\beta }\circ
f)(0)(1),d^{2}(\psi _{\alpha }\circ \psi _{\beta }^{-1}\circ \psi _{\beta
}\circ f)(0)(1,1))+
\end{equation*}%
\begin{equation*}
\Gamma _{\alpha }((\psi _{\alpha }\circ \psi _{\beta }^{-1})(\psi _{\beta
}(x)))(d(\psi _{\alpha }\circ \psi _{\beta }^{-1}\circ \psi _{\beta }\circ
f)(0)(1))[d(\psi _{\alpha }\circ \psi _{\beta }^{-1}\circ \psi _{\beta
}\circ f)(0)(1)]=
\end{equation*}%
\begin{equation*}
(\sigma _{\alpha \beta }(\psi _{\beta }(x)),d\sigma _{\alpha \beta }(\psi
_{\beta }(x))(u),d\sigma _{\alpha \beta }(\psi _{\beta }(x))(d^{2}(\psi
_{\beta }\circ f)(0)(1,1))
\end{equation*}%
\begin{equation*}
+d^{2}\sigma _{\alpha \beta }(\psi _{\beta }(x))(u)[u]+\Gamma _{\alpha
}(\sigma _{\alpha \beta }(\psi _{\beta }(x)))(d\sigma _{\alpha \beta }(\psi
_{\beta }(x))(u))[d\sigma _{\alpha \beta }(\psi _{\beta }(x))(u)])=
\end{equation*}%
\begin{equation*}
(\sigma _{\alpha \beta }(\psi _{\beta }(x)),d\sigma _{\alpha \beta }(\psi
_{\beta }(x))(u),d\sigma _{\alpha \beta }(\psi _{\beta }(x))(d^{2}(\psi
_{\beta }\circ f)(0)(1,1)+\Gamma _{\beta }(\psi _{\beta }(x))(u)[u])=
\end{equation*}%
\begin{equation*}
=(\sigma _{\alpha \beta }(\psi _{\beta }(x)),d\sigma _{\alpha \beta }(\psi
_{\beta }(x))(u),d\sigma _{\alpha \beta }(\psi _{\beta }(x))(v)),
\end{equation*}%
where by $\sigma _{\alpha \beta }$ we denote again the diffeomorphisms $\psi
_{\alpha }\circ \psi _{\beta }^{-1}$. Therefore, the restrictions to the
fibres
\begin{equation*}
\Phi _{\alpha ,x}\circ \Phi _{\beta ,x}^{-1}:\mathbb{E}\times \mathbb{%
E\rightarrow E}\times \mathbb{E:}(u,v)\longmapsto (\Phi _{\alpha }\circ \Phi
_{\beta }^{-1})|_{\pi _{2}^{-1}(x)}(u,v)
\end{equation*}%
are linear isomorphisms and the mappings:%
\begin{equation*}
T_{\alpha \beta }:U_{\alpha }\cap U_{\beta }\rightarrow \mathcal{L}(\mathbb{E%
}\times \mathbb{E},\mathbb{E}\times \mathbb{E}):x\longmapsto \Phi _{\alpha
,x}\circ \Phi _{\beta ,x}^{-1}
\end{equation*}%
are smooth since $T_{\alpha \beta }=(d\sigma _{\alpha \beta }\circ \psi
_{\beta })\times (d\sigma _{\alpha \beta }\circ \psi _{\beta })$ holds for
each $\alpha ,\beta \in I$.

As a result, $T^{2}M$ is a vector bundle over $M$ with fibres of type $%
\mathbb{E}\times \mathbb{E}$ and structure group $GL(\mathbb{E}\times
\mathbb{E})$. Moreover, $T^{2}M$ is isomorphic to $TM\times TM$ since both
bundles are characterized by the same cocycle $\{(d\sigma _{\alpha \beta
}\circ \psi _{\beta })\times (d\sigma _{\alpha \beta }\circ \psi _{\beta
})\}_{\alpha ,\beta \in I}$ of transition functions. \bigskip
\end{proof}
The converse of the theorem was proved also in~\cite{Dod-Gal1}.
These results coincide in the finite dimensional case with the earlier result
since the corresponding transition functions are
identical (see~\cite{Dod-Rad} Corollary 2).

The finite dimensional results~\cite{Dod-Rad,DodT} on
the \textit{frame bundle} \textit{of order two}
\begin{equation*}
L^{2}(M):=\underset{x\in M}{\cup }\mathcal{L}is(\mathbb{E}\times \mathbb{E}%
,T_{x}^{2}M),
\end{equation*}
were extended also to the Banach manifold $M$ by Dodson and Galanis~\cite{Dod-Gal2}:
\begin{theorem}
Every linear connection $\nabla$ of the second order tangent bundle $T^{2}M$
corresponds bijectively to a connection $\omega $ of $L^{2}(M)$.
\end{theorem}

\section{Fr\'{e}chet second tangent bundle}
Let $\mathbb{F}_{1}$ and $\mathbb{F}_{2}$ be two \emph{Hausdorff locally convex topological vector spaces}, and let $U$ be an open subset of $\mathbb{F}_{1}$.
A continuous map $f: U\rightarrow \mathbb{F}_{2}$ is called
{\em differentiable at $x\in U$} if  there exists a continuous linear map $Df(x):\mathbb{F}_{1}\rightarrow \mathbb{F}_{2}$ \label{Df(x)} such that
\[
R(t,v):=
\begin{cases}
  \begin{matrix}
     \frac{1}{t}\left(f(x+tv)-f(x)-Df(x)(tv)\right)\,\raisebox{.4ex}{,} & t\neq 0 \\
         0,                            & t=0
  \end{matrix}
\end{cases}
\]
is continuous at every $(0,v)\in \mathbb{R\times F}_{1}$. The map $f$ will be said to be {\em differentiable} if it is differentiable at every $x\in U$. We call $Df(x)$ the {\em differential} (or {\em derivative}) {\em of $f$ at $x$}.
As in classical (Fr\'echet) differentiation, $Df(x)$ is uniquely determined, see Leslie~\cite{LE1} and \cite{LE2} for more details.

A map $f: U\rightarrow \mathbb{F}_{2}$, as before, is called {\em $C^{1}$-differentiable} if it is differentiable at every point $x\in U$,
and the {\em (total) differential} or {\em (total) derivative} \label{totdif}
\[
  Df: U\times \mathbb{F}_{1}\rightarrow \mathbb{F}_{2}:
                                     (x,v)\mapsto Df(x)(v)
\]
is continuous.

This total differential $Df$ does not involve the space of continuous linear
maps $\mathcal{L}(\mathbb{F}_1,\mathbb{F}_2)$, thus avoiding the possibility of dropping out of the working category when $\mathbb{F}_{1}$ and $\mathbb{F}_{2}$ are Fr\'echet spaces.
The notion of $C^{n}$-differentiability ($n\geq 2$) can be defined by induction and $C^{\infty}$-differentiability follows.

Using the methodology of Galanis and Vassiliou~\cite{Gal2,VG 1} for tangent and frame bundles,
 a vector bundle structure was obtained on the second order
tangent bundles for those Fr\'{e}chet manifolds which can
be obtained as projective limits of Banach manifolds~\cite{Dod-Gal1}.
Let $M$ be a smooth manifold modeled on the Fr\'{e}chet space $\mathbb{F}.$
Taking into account that the latter \emph{always} can be
realized as a projective limit of Banach spaces $\{\mathbb{E}^{i};\rho
^{ji}\}_{i,j\in \mathbb{N}}$ (i.e. $\mathbb{F\cong }\varprojlim \mathbb{E}%
^{i}$) we assume that the manifold itself is obtained as the limit of a
projective system of Banach modeled manifolds $\{M^{i};\varphi
^{ji}\}_{i,j\in \mathbb{N}}.$
Then, it was proved~\cite{Dod-Gal1} that
the second order tangent bundles $\{T^{2}M^{i}\}_{i\in \mathbb{N}}$ form
also a projective system with limit (set-theoretically) isomorphic to
$T^{2}M .$
We define a vector bundle structure on $T^{2}M $ by means of a certain type
of linear connection on $M.$ The
problems concerning the structure group of this bundle
are overcome by the replacement of the pathological $GL(\mathbb{F}\times
\mathbb{F})$ by the new topological (and in a generalized sense smooth Lie)
group:
\begin{equation*}
\mathcal{H}^{0}(\mathbb{F\times F}):=\{(l^{i})_{i\in \mathbb{N}}\in {%
\prod_{i=1}^{\infty }}GL(\mathbb{E}^{i}\mathbb{\times E}^{i}):\,\varprojlim
l^{i}\,\text{\ exists}\}.
\end{equation*}%
Precisely, $\mathcal{H}^{0}(\mathbb{F\times F})$ is a topological
group that is isomorphic to the projective limit of the Banach-Lie groups
\begin{equation*}
\mathcal{H}_{i}^{0}(\mathbb{F\times F}):=\{(l^{1},l^{2},...,l^{i})_{i\in
\mathbb{N}}\in {\prod_{k=1}^{i}}GL(\mathbb{E}^{k}\mathbb{\times E}%
^{k}):\,\rho^{jk}\circ l^{j}=l^{k}\circ \rho^{jk}\,\text{\ }(k\leq j\leq
i)\}.
\end{equation*}%
Also, it can be considered as a generalized Lie group via its
embedding in the topological vector space $\mathcal{L}(\mathbb{F\times F})$.
\begin{theorem}
If a Fr\'{e}chet manifold $M=\varprojlim M^{i}$ is endowed with a linear
connection $\nabla$ that can be realized also as a projective limit of
connections $\nabla=\varprojlim \nabla^{i}$, then $T^{2}M$ is a Fr\'{e}chet vector
bundle over $M$ with structure group $\mathcal{H}^{0}(\mathbb{F\times F}).$
\end{theorem}
\begin{proof}
Following the terminology established above, we consider
$\{(U_{\alpha }=\varprojlim U_{\alpha }^{i},\psi _{\alpha }=\varprojlim \psi
_{\alpha }^{i})\}_{\alpha \in I}$ an atlas of $M$. Each linear connection 
$\nabla^{i}$ $(i\in \mathbb{N)}$, which is naturally associated to a family of
Christoffel symbols $\{\Gamma _{\alpha }^{i}:\psi _{\alpha }^{i}(U_{\alpha
}^{i})\rightarrow \mathcal{L}(\mathbb{E}^{i},\mathcal{L}(\mathbb{E}^{i},%
\mathbb{E}^{i}))\}_{\alpha \in I}$, ensures that $T^{2}M^{i}$ is a vector
bundle over $M^{i}$ with fibres of type $\mathbb{E}^{i}$. This structure, as
already presented in Theorem~\ref{T2Mvb}, is defined by the trivializations:%
\begin{equation*}
\Phi _{\alpha }^{i}:(\pi _{2}^{i})^{-1}(U_{\alpha }^{i})\longrightarrow
U_{\alpha }^{i}\times \mathbb{E}^{i}\times \mathbb{E}^{i},
\end{equation*}%
with
\begin{equation*}
\Phi _{\alpha }^{i}([f,x]_{2}^{i})=(x,(\psi _{\alpha }^{i}\circ f)^{\prime
}(0),(\psi _{\alpha }^{i}\circ f)^{\prime \prime }(0)+\Gamma _{\alpha
}^{i}(\psi _{\alpha }^{i}(x))((\psi _{\alpha }^{i}\circ f)^{\prime
}(0))[(\psi _{\alpha }^{i}\circ f)^{\prime }(0)]);\ \alpha \in I.
\end{equation*}

The families of mappings $\{g^{ji}\}_{i,j\in
\mathbb{N}}$, $\{\varphi ^{ji}\}_{i,j\in \mathbb{N}}$, $\{\rho
^{ji}\}_{i,j\in \mathbb{N}}$ are connecting morphisms of the projective
systems $T^{2}M=\varprojlim (T^{2}M^{i})$, $M=\varprojlim M^{i}$, $\mathbb{F=%
}\varprojlim \mathbb{E}^{i}$ respectively. These projections
$\{\pi _{2}^{i}:T^{2}M^{i}\rightarrow M^{i}\}_{i\in \mathbb{N}}$ satisfy
\begin{equation*}
\varphi ^{ji}\circ \pi _{2}^{j}=\pi _{2}^{i}\circ g^{ji}\text{ \ }(j\geq i%
\mathbb{)}
\end{equation*}
and the trivializations $\{\Phi _{\alpha }^{i}\}_{i\in \mathbb{N}}$
\begin{equation*}
(\varphi ^{ji}\times \rho ^{ji}\times \rho ^{ji})\circ \Phi _{\alpha
}^{j}=\Phi _{\alpha }^{i}\circ g^{ji}\text{ \ }(j\geq i\mathbb{).}
\end{equation*}
We obtain the surjection $\pi _{2}=\varprojlim \pi _{2}^{i}:T^{2}M\longrightarrow M$
and,
\begin{equation*}
\Phi _{\alpha }=\varprojlim \Phi _{\alpha }^{i}:\pi _{2}^{-1}(U_{\alpha
})\longrightarrow U_{\alpha }\times \mathbb{F}\times \mathbb{F}\text{ \ }
(\alpha \in I)
\end{equation*}
is smooth, as a projective limit of smooth mappings, and its projection to
the first factor coincides with $\pi _{2}$.
The restriction to a fibre $\pi _{2}^{-1}(x)$ of $\Phi _{\alpha }$
is a bijection since $\Phi _{\alpha ,x}:=pr_{2}\circ \Phi
_{\alpha }|_{\pi _{2}^{-1}(x)}=\varprojlim (pr_{2}\circ \Phi _{\alpha
}^{i}|_{(\pi _{2}^{i})^{-1}(x)})$.

The corresponding
transition functions $\{T_{\alpha \beta }=\Phi _{\alpha ,x}\circ \Phi
_{\beta ,x}^{-1}\}_{\alpha ,\beta \in I}$  can be considered as taking
values in the generalized Lie group $\mathcal{H}^{0}(\mathbb{F}\times
\mathbb{F)}$, since $T_{\alpha \beta }=\epsilon \circ T_{\alpha \beta
}^{\ast }$, where $\{T_{\alpha \beta }^{\ast }\}_{\alpha ,\beta \in I}\ $
are the smooth mappings
\begin{equation*}
T_{\alpha \beta }^{\ast }:U_{\alpha }\cap U_{\beta }\rightarrow \mathcal{H}%
^{0}(\mathbb{F}\times \mathbb{F}):x\longmapsto (pr_{2}\circ \Phi _{\alpha
}^{i}|_{(\pi _{2}^{i})^{-1}(x)})_{i\in \mathbb{N}}
\end{equation*}
with $\epsilon $ the natural inclusion
\begin{equation*}
\epsilon :\mathcal{H}^{0}(\mathbb{F}\times \mathbb{F})\rightarrow \mathcal{L}
(\mathbb{F\times F}):(l^{i})_{i\in \mathbb{N}}\longmapsto \varprojlim l^{i}.
\end{equation*}
Hence, $T^{2}M$ admits a vector bundle structure
over $M$ with fibres of type $\mathbb{F}\times \mathbb{F}$ and structure
group $\mathcal{H}^{0}(\mathbb{F}\times \mathbb{F})$. This bundle
is isomorphic to $TM\times TM$ since they have identical
transition functions:
\begin{equation*}
T_{\alpha \beta }(x)=\Phi _{\alpha ,x}\circ \Phi _{\beta ,x}^{-1}=(d(\psi
_{a}\circ \psi _{\beta }^{-1})\circ \psi _{\beta })(x)\times (d(\psi
_{a}\circ \psi _{\beta }^{-1})\circ \psi _{\beta })(x)
\end{equation*}
\end{proof}
Also, the converse is true:
\begin{theorem}
If $T^{2}M$ is an $\mathcal{H}^{0}(\mathbb{F}\times \mathbb{F})-$Fr\'{e}chet
vector bundle over $M$ isomorphic to $TM\times TM$, then $M$ admits a linear
connection which can be realized as a projective limit of connections.
\end{theorem}

\section{Fr\'{e}chet second frame bundle}
Let $M=\varprojlim M^{i}$ be a manifold with connecting morphisms $\{\varphi
^{ji}:M^{j}\rightarrow M^{i}\}_{i,j\in \mathbb{N}}$ and Fr\'{e}chet space model the
limit $\mathbb{F}$ of a projective system of Banach spaces 
$\{\mathbb{F}^{i};\rho ^{ji}\}_{i,j\in \mathbb{N}}$. Following the results obtained in
\cite{Dod-Gal1}, if $M$ is endowed with a linear
connection $\nabla=\varprojlim \nabla^{i}$, then $T^{2}M$ admits a vector bundle
structure over $M$ with fibres of Fr\'{e}chet type $\mathbb{F}\times \mathbb{F}$. Then $T^{2}M$ becomes also a projective limit of manifolds
via the identification $T^{2}M\simeq \varprojlim T^{2}M^{i}$.

Let
\begin{equation*}
\mathcal{F}^{2}M^{i}=\underset{x^{i}\in M^{i}}{\cup }%
\{(h^{k})_{k=1,...,i}:h^{k}\in \mathcal{L}is(\mathbb{F}^{k}\times \mathbb{F}%
^{k},T_{\varphi ^{ik}(x^{i})}^{2}M^{k})\text{ and }
\end{equation*}%
\begin{equation*}
g^{mk}\circ h^{m}=h^{k}\circ (\rho ^{mk}\times \rho ^{mk}),\text{ }i\geq
m\geq k\}.
\end{equation*}
We replace the pathological
general linear group $GL(\mathbb{F})$ by
\begin{equation*}
H_{0}(\mathbb{F}):=H_{0}(\mathbb{F},\mathbb{F})=\{(l^{i})_{i\in \mathbb{N}%
}\in \prod_{i=1}^{\infty }GL(\mathbb{F}^{i}):\,\,\varprojlim l^{i}\,\text{\
exists}\}.
\end{equation*}%
The latter can be thought of also as a generalized Fr\'{e}chet Lie group by
being embedded in $H(\mathbb{F}):=H(\mathbb{F},\mathbb{F})$.
Then~\cite{DGV 1},
\begin{theorem}$\mathcal{F}^{2}M^{i}$ is a principal fibre bundle over $M^{i}$ with
structure group the Banach Lie group $H_{0}^{i}(\mathbb{F\times F}%
):=H_{0}^{i}(\mathbb{F\times F},\mathbb{F\times F})$.\\
The limit $\varprojlim \mathcal{F}^{2}M^{i}$ is a Fr\'{e}chet principal
bundle over $M$ with structure group $H_{0}(\mathbb{F\times F})$.
\end{theorem}
We call the \textit{generalized bundle of frames} \textit{of order two} of
the Fr\'{e}chet manifold $M=\varprojlim M^{i}$ the principal bundle
\begin{equation*}
\mathcal{F}^{2}(M):=\varprojlim \mathcal{F}^{2}M^{i}.
\end{equation*}
This is a natural generalization of the usual frame bundle and it follows
\begin{theorem}
For the action of the group $H^{0}(\mathbb{F}\times \mathbb{F})$ on the
right of the product $\mathcal{F}^{2}(M)\times (\mathbb{F}\times \mathbb{F}):$
\begin{equation*}
((h^{i}),(u^{i},v^{i}))_{i\in \mathbb{N}}\cdot (g^{i})_{i\in \mathbb{N}%
}=((h^{i}\circ g^{i}),(g^{i})^{-1}(u^{i},v^{i}))_{i\in \mathbb{N}},
\end{equation*}%
the quotient space $\mathcal{F}^{2}M\times (\mathbb{F}\times \mathbb{F}
)\diagup H^{0}(\mathbb{F}\times \mathbb{F})$ is isomorphic with $T^{2}M$.
\end{theorem}

Consider a connection of $\mathcal{F}^{2}(M)$
represented by the 1-form $\omega \in \Lambda ^{1}(\mathcal{F}^{2}(M),%
\mathcal{L}(\mathbb{F}\times \mathbb{F}))$, with smooth atlas $\{(U_{\alpha
}=\varprojlim U_{\alpha }^{i},\psi _{\alpha }=\varprojlim \psi _{\alpha
}^{i})\}_{a\in I}$ of $M$, $\{(p^{-1}(U_{\alpha }),\Psi _{\alpha })\}_{a\in
I}$  trivializations of $\mathcal{F}^{2}(M)$ and $\{\omega
_{\alpha }:=s_{\alpha }^{\ast }\omega \}_{a\in I}$ the corresponding local
forms of $\omega $ obtained as pull-backs with respect to the natural local
sections $\{s_{\alpha }\}$ of $\{\Psi _{\alpha }\}.$ Then a (unique) linear
connection can be defined on $T^{2}M$ by means of the Christoffel symbols
\begin{equation*}
\Gamma _{\alpha }:\psi _{\alpha }(U_{\alpha })\rightarrow \mathcal{L}(%
\mathbb{F}\times \mathbb{F},\mathcal{L}(\mathbb{F},\mathbb{F}\times \mathbb{%
F))}
\end{equation*}%
with $([\Gamma _{\alpha }(y)](u))(v)=\omega _{\alpha }(\psi _{\alpha
}^{-1}(y))(T_{y}\psi _{\alpha }^{-1}(v))(u),$ $(y,u,v)\in \psi _{\alpha
}(U_{\alpha })\times \mathbb{F}\times \mathbb{F}\times \mathbb{F}$.

However, in the framework of Fr\'{e}chet bundles an arbitrary connection is
not always easy to handle, since Fr\'{e}chet manifolds and bundles lack
a general theory of solvability for linear differential equations. Also,
Christoffel symbols (in the case of vector
bundles) or the local forms (in principal bundles) are affected in their
representation of linear maps by
the fact that continuous linear mappings of a Fr\'{e}chet space
do not remain in the same category.
Galanis~\cite{Gal2,Gal3} solved the problem for connections
that can be obtained as projective limits and we obtain~\cite{DGV 1}
\begin{theorem}
Let $\nabla $ be a linear connection of the second order tangent bundle $%
T^{2}M=\varprojlim T^{2}M^{i}$ that can be represented as a projective limit
of linear connections $\nabla ^{i}$ on the (Banach modelled) factors. Then $%
\nabla $ corresponds to a connection form $\omega $ of $\mathcal{F}^{2}M$
obtained also as a projective limit.
\end{theorem}
Areas of application were outlined in~\cite{DGV 1}.

\section{Connection choice}
Dodson, Galanis and Vassiliou~\cite{DGV 2} studied the way in which the choice of connection influenced the structure of the second tangent bundle over Fr\'{e}chet manifolds, since each connection determines one isomorphism of $T^2M\equiv TM\bigoplus TM.$ They defined the second order differential $T^{2}f$ of a smooth map
$g:M\rightarrow N$ between two manifolds $M$ and $N$. In contrast
to the case of the first order differential $Tg$, the linearity of
$T^{2}g$ on the fibres ($T_{x}^{2}g:T_{x}^{2}M\rightarrow T_{g(x)}^{2}N$, $x\in M$)
is not always ensured but they proved a number of results.

The connections $\nabla _{M}$ and $\nabla _{N}$ are called $g$\emph{-conjugate}~\cite{Vass}
(or $g$\emph{-related}) if they commute with the differentials of $g:$
\begin{equation}\label{eq1}
Tg\circ \nabla _{M}=\nabla _{N}\circ T(Tg).
\end{equation}
Locally
\begin{equation}\label{eq2}
 \begin{gathered}
   Tg(\phi _{\alpha }(x))(\Gamma _{\alpha }^{M}(\phi _{\alpha
   }(x))(u)(u))=\\
 \Gamma _{\beta }^{N}(g(\phi _{\alpha }(x)))(Tg(\phi _{\alpha
}(x))(u))(Tg(\phi _{\alpha }(x))(u))+T(Tg)((\phi _{\alpha
}(x))(u,u),
   \end{gathered}
\end{equation}
for every $(x,u)\in U_{\alpha }\times \mathbb{E}.$
For $g$-conjugate connections $\nabla _{M}$ and $\nabla _{N}$ the local expression of $T_{x}^{2}g$ reduces to
\begin{equation}
(\Psi _{\beta ,g(x)}\circ T_{x}^{2}g\circ \Phi
_{a,x}^{-1})(u,v)=(DG(\phi _{\alpha }(x))(u),DG(\phi _{\alpha
}(x))(v)).    \label{eq3}
\end{equation}
\begin{theorem}\label{thm2.5}
Let $T^{2}M$, $T^{2}N$ be the second order tangent bundles defined
by the pairs $(M,\nabla _{M})$, $(N,\nabla _{N})$, and let
$g:M\rightarrow N$
be a smooth map. If the connections $\nabla _{M}$ and $\nabla _{N}$ are $%
g$-conjugate, then the second order differential
$T^{2}g:T^{2}M\rightarrow T^{2}N$ is a vector bundle morphism.
\end{theorem}
\begin{theorem}
Let $\nabla $, $\nabla ^{\prime }$ be two linear connections on
$M$. If $g$
is a diffeomorphism of $M$ such that $\nabla $ and $\nabla ^{\prime }$ are $%
g $-conjugate, then the vector bundle structures on $T^{2}M$, induced by $%
\nabla $ and $\nabla ^{\prime }$, are isomorphic. \label{isom}
\end{theorem}

\section{Differential equations}
The importance of Fr\'{e}chet manifolds arises from their ubiquity as quotient spaces of
bundle sections and hence as environments for differential equations on such spaces. This context was addressed next
in~\cite{Aghasi 1} and those authors provided a new way of
representing and solving a wide class of evolutionary equations
on Fr\'{e}chet manifolds of sections.

First~\cite{Aghasi 1} considered a Banach manifold $M,$ and defined
an \emph{integral curve} of $\xi $ as a smooth map $\theta:J\rightarrow M$,
defined on an open interval $J$ of $\mathbb{R}$, if it satisfies the
condition
\begin{equation}
T_{t}^{2}\theta (\partial _{t})=\xi (\theta (t)).  \label{2.1}
\end{equation}%
Here $\partial _{t}$ is the second order tangent vector of $T_{t}^{2}%
\mathbb{R}$ induced by a curve $c:\mathbb{R}\rightarrow \mathbb{R}$ with $%
c^{\prime }(0)=1,c^{\prime \prime }(0)=1$.
If $M$
is simply a Banach space $\mathbb{E}$ with differential structure induced by the
global chart $(\mathbb{E},id_{\mathbb{E}})$, then the generalization is clear since the
above condition reduces to the second derivative of $\theta $:
\begin{equation*}
T_{t}^{2}\theta (\partial _{t})=\theta ^{\prime \prime }(t)=D^{2}\theta
(t)(1,1).
\end{equation*}
Then the following were proved~\cite{Aghasi 1}.
\begin{theorem}
\label{th22} Let $\xi $ be a second order vector field on a manifold $M$ modeled on Banach space $\mathbb{E}.$
Then, the existence of an integral curve $\theta $ of $\xi $ is equivalent to the
solution of a system of second order differential equations on $\mathbb{E}$.
\end{theorem}
Of course, these
second order differential equations depend not
only on the choice of the second order vector field but also the choice of the linear
connection that underpins the vector bundle structure.
In the case of a Banach manifold that is a Lie group, $M=(G,\gamma),$
\begin{theorem}
Let $v$ be any vector of the second order tangent space of $G$ over the
unitary element. Then, a corresponding left invariant second order vector
field $\xi$ of $G$ may be constructed. Also,
every monoparametric subgroup $\beta :\mathbb{R}\rightarrow G$ is an
integral curve of the second order left invariant vector field $\xi ^{2}$ of
$G$ that corresponds to $\ddot{\beta}(0)$.
\end{theorem}

Extending this to a Fr\'{e}chet manifold $M$ that is the projective limit of Banach manifolds~\cite{Dod-Gal1}, yielded the result:
\begin{theorem}
Every second order vector field $\xi $ on $M$ obtained as projective limit
of second order vector fields $\{\xi ^{i}$ on $M^{i}\}_{i\in \mathbb{N}}$
admits locally a unique integral curve $\theta $ ~satisfying an initial
condition of the form $\theta (0)=x~$and$~T_{t}\theta (\partial _{t})=y$, $x$
$\in M,$ $y\in T_{\theta (t)}M$, provided that the components $\xi ^{i}$
admit also integral curves of second order.
\end{theorem}

\section{Hypercyclicity}
A continuous operator $T$ on a topological vector space $\mathbb{E}$ is {\em cyclic} if for some $f\in \mathbb{E}$ the span of $\{T^nf, n\geq 0 \}$ is dense in $\mathbb{E}.$ Also, $T$ is {\em hypercyclic} if, for some $f,$ called a {\em hypercyclic vector}, $\{T^nf, n\geq 0 \}$ is dense in $\mathbb{E},$  and {\em supercyclic} if the projective space orbit
$\{ \lambda T^nf, \lambda \in \mathbb{C}, n\geq 0 \}$ is dense in $\mathbb{E}.$ These properties are called {\em weakly hypercyclic}, {\em weakly supercyclic} respectively, if $T$  has the property with respect to the weak topology.
For example, the translation by a fixed nonzero $z\in \mathbb{C}$ is hypercyclic on the Fr\'{e}chet space $\mathbb{H}(\mathbb{C})$ of entire functions, and so is the differentiation operator $f\mapsto f'.$ Any power $T^m$ of a hypercyclic linear operator is hypercyclic, Ansari~\cite{Ansari}. Finite dimensional spaces do not admit hypercyclic operators, Kitai~\cite{Kitai}.

More generally, a sequence of linear operators $\{T_n\}$ on a topological is called hypercyclic if, for some $f\in \mathbb{E},$ the set $\{T_nf, n\in \mathbb{N}\}$ is dense in $\mathbb{E};$ see Chen and Shaw~\cite{ChenShaw} for a discussion of related properties. The sequence $\{T_n\}$ is said to satisfy the {\em Hypercyclicity Criterion} for an increasing sequence $\{n(k)\}\subset \mathbb{N}$ if there are dense subsets $X_0,Y_0\subset \mathbb{E}$ satisfying:
\begin{eqnarray*}
   &(\forall f\in X_0)& T_{n(k)} f\rightarrow 0 \nonumber \\
   &(\forall g\in Y_0)& {\rm there \ is \ a \ sequence} \ \{u(k)\}\subset \mathbb{E} \
    {\rm such \ that} \ u(k) \rightarrow 0 \ {\rm and} \ T_{n(k)} u(k)\rightarrow 0. \nonumber
\end{eqnarray*}
Bes and Peris~\cite{BesPeris} proved that on a  separable Fr\'{e}chet space $\mathbb{F}$ a continuous linear operator $T$ satisfies the Hypercyclicity Criterion if and only if $T\oplus T$ is hypercyclic on $\mathbb{F}\oplus\mathbb{F}.$  Moreover, if $T$ satisfies the Hypercyclicity Criterion then so does every power $T^n$ for $n\in \mathbb{N}.$

The book by Bayart and Matheron~\cite{BayartMath} provides more details of the theory of hypercyclic operators. Berm\'{u}dez et al.~\cite{BermudezB&C} investigated hypercyclicity, topological mixing and chaotic maps on Banach spaces. Bernal and Grosse-Erdmann studied the existence of hypercyclic semigroups of continuous operators on a Banach space. Albanese et al.~\cite{ABR} considered cases when it is possible to  extend Banach space results on $C_0$-semigroups of continuous linear operators to  Fr\'{e}chet spaces. Every operator norm continuous semigroup in a Banach space $X$ has an infinitesimal generator belonging to the space of continuous linear operators on $X;$ an example is given to show that this fails in a general Fr\'{e}chet space. However, it does not fail for countable products of Banach spaces and quotients of such products; these are the Fr\'{e}chet spaces that are quojections, the projective sequence consisting of surjections. Examples include the sequence space $\mathbb{C}^\mathbb{N}$ and the Fr\'{e}chet space of continuous functions $C(X)$ with $X$ a $\sigma$-compact completely regular topological space and compact open topology.

Grosse-Erdmann~\cite{GrE} related hypercyclicity to the topological universality concept, and showed that an operator $T$ is hypercyclic on a separable Fr\'{e}chet space $\mathbb{F}$ if it has the {\em topological transitivity property}: for every pair of nonempty open subsets $U,V\subseteq \mathbb{F}$ there is some $n\in \mathbb{N}$ such that $T^n(U)\bigcap V\neq \emptyset.$ Chen and Shaw~\cite{ChenShaw} related hypercyclicity to topological mixing, following Costakis and Sambarino~\cite{CS} who showed that if $T^n$ satisfies the Hypercyclicity Criterion then $T$ is {\em topologically mixing} in the sense that:
for every pair of nonempty open subsets $U,V\subseteq \mathbb{F}$ there is some $N\in \mathbb{N}$ such that $T^n(U)\bigcap V\neq \emptyset$ for all $n\geq N.$ See also Berm\'{u}dez et al.~\cite{BermudezB&C} for further studies of hypercyclic and chaotic maps on Banach spaces in the context of topological mixing.

It was known that the direct sum of two hypercyclic operators need not be hypercyclic but recently De La Rosa and Read~\cite{delaRosa} showed that even the direct sum of a hypercyclic operator with itself $T\oplus T$ need not be hypercyclic.
Bonet and Peris~\cite{BonetPeris} showed that every separable
infinite dimensional Fr\'{e}chet space $\mathbb{F}$ supports a hypercyclic operator.
Moreover, from Shkarin~\cite{Shkarin},
there is a linear operator $T$ such that the direct sum
$T\oplus T\oplus .. .\oplus T = T^{\oplus m}$
of $m$ copies of $T$ is a hypercyclic operator on $\mathbb{F}^m$ for each $m\in \mathbb{N}.$
An $m$-tuple $(T,T,...,T)$ is called {\em disjoint
hypercyclic} if there exists $f\in \mathbb{F}$ such that $(T_1^nf,T_2^nf,...,T_m^nf), n=1,2,...$ is dense in $\mathbb{F}^m.$ See Salas~\cite{Salas} and Bernal-Gonz\'{a}lez~\cite{Bernal} for examples and recent results.

O'Regan and Xian~\cite{OReganX} proved fixed point theorems for maps and multivalued maps between Fr\'{e}chet spaces, using projective limits and the classical Banach theory. Further recent work on set valued maps between Fr\'{e}chet spaces can be found in Galanis et al.{\cite{GBL,GBLP,ORegan} and Bakowska and Gabor~\cite{BakoG}.

Montes-Rodriguez et al.~\cite{MRRM&S} studied the Volterra composition operators $V_\varphi$ for $\varphi$ a measurable self-map of $[0,1]$ on functions $f\in L^p[0,1], \ 1\leq p \leq \infty$
\begin{equation}\label{Volterra}
    (V_\varphi f)(x) =\int_0^\varphi(x) f(t) dt
\end{equation}
These operators generalize the classical Volterra operator $V$ which is the case when $\varphi$ is the identity. $V_\varphi$ is measurable, and compact on $L^p[0,1].$

Consider the Fr\'{e}chet space  $\mathbb{F}=C_0[0,1),$ of continuous functions vanishing at zero with the topology of uniform convergence on compact subsets of $[0,1).$ It was known that the action of $V_\varphi$ on $C_0[0,1)$ is hypercyclic when $\varphi(x)=x^b, b\in (0,1)$~\cite{Herzog}. This result has now been extended by Montes-Rodriguez et al. to give the following complete characterization.
 \begin{theorem}~\cite{MRRM&S}
For $\varphi\in C_0[0,1)$ the following are equivalent\\
{\bf (i)} $\varphi$ is strictly increasing with $\varphi(x)>x$ for $x\in (0,1)$
{\bf (ii)}  $V_\varphi$ is weakly hypercyclic
{\bf (iii)} $V_\varphi$ is hypercyclic.
\end{theorem}
Karami et al~\cite{Karami} seem to obtain examples of hypercyclic operators on $H_{bc}(\mathbb{E}),$ the space of bounded functions on compact subsets of Banach space $\mathbb{E}.$ For example, when  $\mathbb{E}$  has separable dual $\mathbb{E}^*$ then for nonzero $\alpha\in\mathbb{E},$ $T_\alpha:f(x)\mapsto f(x+\alpha)$ is hypercyclic.
As for other cases of hypercyclic operators on Banach spaces, it would be interesting to know when the property persists to projective limits of the domain space.

Yousefi and Ahmadian~\cite{YA} studied the case that $T$ is a continuous linear operator on an infinite dimensional Hilbert space $\mathbb{H}$ and left multiplication is hypercyclic with respect to the strong operator topology. Then there is a Fr\'{e}chet space  $\mathbb{F}$ containing  $\mathbb{H},$ $\mathbb{F}$ is the completion of $\mathbb{H},$ and for every nonzero vector $f\in \mathbb{H}$ the orbit $\{T^nf,n\geq \}$ meets any open subbase of $\mathbb{F}.$

\vspace{0.5cm}
\noindent{\bf Acknowlegement} This review is based on joint work with G. Galanis and E. Vassiliou whose advice is gratefully acknowledged.

\noindent C.T.J. Dodson\\
School of  Mathematics, University of Manchester\\
Manchester M13 9PL UK.\\
ctdodson@manchester.ac.uk
\end{document}